\documentclass[11pt]{amsart}
\usepackage{amssymb}
\usepackage[cmtip,all]{xy}
\DeclareMathOperator{\fin}{fin}
\DeclareMathOperator{\Fin}{fin}
\newcommand{\cstar}{$\textrm{C}^*$}
\newcommand{\set}[2]{\left\{#1\mathrel{}\middle|\mathrel{}#2\right\}}
\newcommand{\seq}[2]{\left\langle #1\;|\; #2\right\rangle}
\newcommand{\fexp}[2]{{}^{#2}{#1}}
\newcommand{\rs}{\mathord{\upharpoonright}}
\newtheorem{theorem}{Theorem}[section]
\newtheorem{lemma}[theorem]{Lemma}
\newtheorem{claim}[theorem]{Claim}
\newtheorem{corollary}[theorem]{Corollary}

\theoremstyle{definition}

\theoremstyle{remark}

\numberwithin{equation}{section}

\title[Homeomorphisms of \v Cech--Stone remainders]{Homeomorphisms of \v
  Cech--Stone Remainders: the Zero-Dimensional Case}
\author{Ilijas Farah}
\address{Ilijas Farah: Department of Mathematics, York University, Toronto, Canada}
\email{ifarah@mathstat.yorku.ca}
\author{Paul McKenney}
\address{Paul McKenney: Department of Mathematics, Miami University, Oxford, Ohio, USA}
\email{mckennp2@miamioh.edu}
\thanks{The first author was partially supported by NSERC}

\begin{document}

  \begin{abstract} We prove, using a weakening of the Proper Forcing Axiom, 
     that any homemomorphism between  \v Cech--Stone remainders of any two
     locally compact, zero-dimensional Polish spaces is induced by a
     homeomorphism between their cocompact subspaces. 
  \end{abstract}

  \maketitle
  
  \section{Introduction}

  The \v Cech--Stone remainder (also known as corona) $\beta X\setminus X$ of a
  topological space $X$ will be denoted~$X^*$.  A continuous map $\varphi : X^*\to
  Y^*$ is called \emph{trivial} if there is a continuous $e : X\to Y$ such that
  $\varphi = e^*$, where $e^* = \beta e \setminus e$ and $\beta e$ is the unique
  continuous extension of $e$ to $\beta X$.  It follows that two remainders
  $X^*$ and $Y^*$ are homeomorphic via a trivial map if and only if there are
  cocompact subspaces of $X$ and $Y$ which are themselves homeomorphic.  In this
  paper we prove the following (see \S\ref{sec:Notation} for the definitions). 
  
  \begin{theorem} \label{main}
     OCA and MA$_{\aleph_1}$ together imply that    every homeomorphism between
     \v Cech--Stone remainders of  locally compact, zero-dimensional, Polish
     spaces is trivial.
  \end{theorem}

  This proves a special case of the  rigidity conjecture that forcing axioms
  imply all homeomorphisms between \v Cech--Stone remainders of locally compact,
  noncompact Polish spaces are trivial (see \cite{Fa:Rigidity}, \cite{Fa:AQ},
  \cite{CoFa:Automorphisms}).  In contrast, the  Continuum Hypothesis (CH),
  implies that  \v Cech--Stone remainders of locally compact, noncompact,
  zero-dimensional Polish spaces are homeomorphic. This is a consequence of
  Parovi\v cenko's topological characterization of $\omega^*$ (see e.g.,
  \cite{vanmill}).  Stone duality between compact, zero-dimensional, Hausdorff
  spaces and Boolean algebras of their clopen sets provides a  model-theoretic
  reformulation of this malleability phenomenon.  For a  locally compact,
  non-compact Hausdorff space $X$ let $\mathcal{C}(X)$ denote the algebra the
  clopen subsets of $X$ and let $\mathcal{K}(X)$ denote its ideal of
  compact-open sets.  If $X$ and $Y$ are in addition zero-dimensional, then
  continuous maps from $X^*$ to $Y^*$ functorially correspond to Boolean algebra
  homomorphisms from $\mathcal{C}(Y)/\mathcal{K}(Y)$ into
  $\mathcal{C}(X)/\mathcal{K}(X)$.  All of these algebras are elementarily
  equivalent and (assuming CH) saturated, and therefore isomorphic   (see
  \cite{DoHa:Universal} for the details and an extension to not necessarily
  zero-dimensional spaces). \footnote{There is a deeper metamathematical
  explanation of the effect of CH; see \cite{Wo:Beyond}.}

  Back to rigidity, Theorem~\ref{main} belongs to a long line of  results going
  back to  Shelah's groundbreaking construction of  an oracle-cc forcing
  extension of the universe in which all autohomeomorphisms of $\omega^*$ are
  trivial (\cite{Sh:Proper}).  Shelah's proof was recast in  terms of forcing
  axioms PFA and OCA+MA$_{\aleph_1}$ in \cite{ShSte:PFA} and
  \cite{Velickovic:OCAA}, respectively.   The latter axiom also implies that
  homeomorphisms between \v Cech--Stone remainders between countable locally
  compact spaces, as well as their arbitrary powers, are trivial (\cite[\S
  4]{Fa:AQ}) as well as strong negations of Parovi\v cenko's theorem
  (\cite{DoHa:Images}, \cite{DoHa:Lebesgue}).

  The interest in quotient rigidity results  was rejuvenated by the discovery
  that the noncommutative analogue of  `are all  automorphisms of $\omega^*$ (or
  of $\mathcal{P}(\omega)/\fin$) trivial?' was a prominent open problem in the theory of
  operator algebras.  Motivated by their work on analytic K-homology, Brown,
  Douglas, and Fillmore asked whether the Calkin algebra associated with the
  separable, infinite-dimensional, complex Hilbert space has  outer
  automorphisms~(\cite{BrDoFi}).  Like its commutative analogue, this question
  cannot be resolved in ZFC, with   CH and OCA implying the opposite answers
  (\cite{PhWe:Calkin}, \cite{Fa:All}).  Other rigidity results in the  setting
  of \cstar-algebras were proved for reduced products of the form $\prod_n
  A_n/\bigoplus_n A_n$ in case when all $A_n$ are matrix algebras
  (\cite{ghasemi2016reduced},  \cite{Ghasemi:FDD}),  separable UHF algebras
  (\cite{McK:Reduced}) or unital separable nuclear  \cstar-algebras
  (\cite{vignati2017thesis}, \cite{MV:N}).

  A general rigidity conjecture for corona \cstar-algebras was stated and
  partially verified in \cite{CoFa:Automorphisms}. The model theory of coronas
  proved to be a bit more complex than that of Boolean algebras.  While the
  reduced products are countably saturated (\cite{FaSh:Rigidity}), coronas
  posses only a modest degree of saturation (\cite{FaHa:Countable},
  \cite{EaVi:Saturation}, \cite{Voi:Countable}, \cite{farah2015calkin}).  In
  return, \cstar-algebras provided a vantage point  that resulted in the
  construction of nontrivial autohomeomorphisms of $X^*$ for every noncompact,
  locally compact, metrizable manifold using CH~(\cite{vignati2017nontrivial}).
  \footnote{The only previously known case was $X=\mathbb R$, see
  \cite{Hart:Cech} and \cite{FaSh:Rigidity}} 

  We note that Theorem~\ref{main} is not optimal. The first author's proof that
  all zero-dimensional, locally compact, Polish spaces satisfy the   \emph{weak
  extension principle} (\cite[Theorem~4.10.1]{Fa:AQ})  will appear elsewhere.
  Dow refuted the related \emph{strong extension principle}
  (\cite[Question~4.11.4]{Fa:AQ}) by constructing a nontrivial continuous map
  from $\omega^*$ into $\omega^*$ (i.e., one that does not have a continuous
  extension to a map from $\beta\omega$ into $\beta\omega$) in ZFC
  (\cite{dow2014non}).  An alternative proof of our main result from a stronger
  assumption (PFA)  is given by~\cite[Theorem~4.3]{FaSh:Rigidity}.  

  In section~\ref{sec:Notation} we introduce some of the language required to
  prove Theorem~\ref{main}.  Section~\ref{sec:Embeddings} treats embeddings of
  $\mathcal{P}(\omega)/\fin$ into $\mathcal{C}(X)/\mathcal{K}(X)$, and we show that under OCA
  +MA$_{\aleph_1}$, every such embedding is trivial.  Much of the proof follows
  the work in \cite{Velickovic:OCAA} and~\cite{Velickovic:D} with only minor
  modifications, so to avoid treading the same ground we only prove one of the
  ingredients going into this theorem.  Section~\ref{sec:Families} completes the
  proof of Theorem~\ref{main} through an analysis of \emph{coherent families of
  continuous functions}.

  \section{Notation}
  \label{sec:Notation}

  Our terminology is standard (see \cite{Ku:Set}).  The assumption of
  Theorem~\ref{main} is a consequence of the \emph{Proper Forcing Axiom}, PFA.
  OCA abbreviates the  \emph{Open Coloring Axiom} (\cite{Todorcevic:PPIT}; not
  to be confused with the eponymous OCA of \cite{ARS}), and MA$_{\aleph_1}$
  refers to Martin's Axiom for $\aleph_1$ dense sets. 

  If $E$ is a set, then $[E]^2$ will denote the set of unordered pairs from $E$.
  If $M\subseteq [E]^2$, then a set $H\subseteq E$ is called
  \emph{$M$-homogeneous} if $[H]^2\subseteq M$.  The \emph{Open Coloring Axiom}
  states: for every separable metric space $E$ and every partition $[E]^2 =
  M_0\cup M_1$ such that $M_0$ is open (here we identify $[E]^2$ with a
  symmetric subset of $E\times E$ minus the diagonal), either
  \begin{enumerate}
    \item  there is an uncountable $M_0$-homogeneous set,  or
    \item  there is a cover of $E$ by countably-many $M_1$-homogeneous sets.
  \end{enumerate}

  We  fix a zero-dimensional, locally compact and noncompact Polish space~$X$.
  Let $\seq{K_n}{n < \omega}$ be an increasing sequence of compact-open sets in
  $X$, such that $X = \bigcup K_n$.  Then $\mathcal{K}(X)$ is generated by $\seq{K_n}{n
  < \omega}$ since
  \[
    K\in\mathcal{K}(X) \iff \exists n\; K\subseteq K_n
  \]
  It is easy to see that $\mathcal{C}(X)$ has size continuum, whereas $\mathcal{K}(X)$ is
  countable.    When $A,B\in\mathcal{C}(X)$ are distinct, we write $\delta(A,B)$ for the
  least $n$ such that $A\cap X_n \neq B\cap X_n$.  If
  \[
    d(A,B) = \left\{\begin{array}{ll} 2^{-\delta(A,B)} & A\neq B \\ 0 & A = B
    \end{array}\right.
  \]
  then $d$ is a Polish metric on $\mathcal{C}(X)$.

  Let $X_0 = K_0$ and $X_{n+1} = K_{n+1}\setminus K_n$.  We will often identify
  $\mathcal{C}(X)$ with $\prod_n \mathcal{C}(X_n)$, and $\mathcal{P}(\omega)$ with $\fexp{2}{\omega}$.
  Under these identifications, $\mathcal{K}(X)$ maps to $\bigoplus_n \mathcal{C}(X_n)$ (the set
  of functions in $\prod_n \mathcal{C}(X_n)$ which are nonempty on only finitely many
  coordinates) and $\fin$ to $\fexp{2}{<\omega}$.  If $Y$ and $Z$ are
  zero-dimensional, locally compact Polish spaces, $\varphi : \mathcal{C}(Y)/\mathcal{K}(Y)\to
  \mathcal{C}(Z)/\mathcal{K}(Z)$ is a homomorphism, and $U\in\mathcal{C}(Y)$, then we write $\varphi\rs U$
  for the restriction $\varphi\rs \mathcal{C}(U)/\mathcal{K}(U)$.  When working with the quotient
  $\mathcal{C}(X)/\mathcal{K}(X)$ we will write $[A]$ for the equivalence class of some
  $A\in\mathcal{C}(X)$.

  \section{Embeddings of $\mathcal{P}(\omega)/\fin$ into $\mathcal{C}(X)/\mathcal{K}(X)$}
  \label{sec:Embeddings}

  Let $e : X\to\omega$ be a continuous map.  If $e^{-1}(n)$ is compact for every
  $n$, then we say $e$ is \emph{compact-to-one}.  If $e$ is compact-to-one, then
  the map $a\mapsto e^{-1}(a)$, from $\mathcal{P}(\omega)$ to $\mathcal{C}(X)$, induces a
  homomorphism $\varphi_e : \mathcal{P}(\omega)/\fin\to \mathcal{C}(X)/\mathcal{K}(X)$.  Moreover, $\varphi_e$
  is injective if and only if $e$ is bounded on compact sets.  We call a
  homomorphism $\varphi : \mathcal{P}(\omega)/\fin\to \mathcal{C}(X)/\mathcal{K}(X)$ \emph{trivial} if it
  is of the form $\varphi_e$ for some compact-to-one, continuous $e$.

  In this section we prove 
  \begin{theorem}
    \label{trivial.embedding}
    Assume OCA+MA$_{\aleph_1}$, and suppose
    \[
      \varphi : \mathcal{P}(\omega)/\fin\to\mathcal{C}(X)/\mathcal{K}(X)
    \]
    is an injective homomorphism. Then $\varphi$ is trivial.
  \end{theorem}
  Towards the proof of Theorem~\ref{trivial.embedding}, we fix an injective
  homomorphism $\varphi : \mathcal{P}(\omega)/\fin\to\mathcal{C}(X)/\mathcal{K}(X)$ and we define 
  \[
    \mathcal{I} = \set{a\subseteq\omega}{\mbox{$\varphi\rs a$ is trivial}}
  \]
  Note that $\mathcal{I}$ is an ideal on $\omega$.

  A family $\mathcal{A}\subseteq\mathcal{P}(\omega)$ is called \emph{almost disjoint} if for
  all distinct $a,b\in\mathcal{A}$, $a\cap b =^* \emptyset$.  Such a family $\mathcal{A}$ is
  called \emph{treelike} if there is some tree $T$ on $\omega$ and a bijection
  $t : \omega\to \fexp{\omega}{<\omega}$ under which each $a\in\mathcal{A}$ corresponds
  to a branch through $T$, and vice-versa.  The following lemma is proven
  in~\cite{Velickovic:OCAA}.
  \begin{lemma}
    \label{ma}
    Assume MA$_{\aleph_1}$.  Then for every uncountable almost-disjoint family
    $\mathcal{A}$ of subsets of $\omega$ we may find an uncountable $\mathcal{B} \subseteq \mathcal{A}$
    and partitions $b = b_0\cup b_1$ for $b\in \mathcal{B}$ such that each family $\mathcal{B}_i
    = \set{b_i}{b\in\mathcal{B}}$ is treelike.
  \end{lemma}
  The following three lemmas do not directly follow from the work
  in~\cite{Velickovic:OCAA}, but their proofs are nearly the same, modulo some
  minor modifications.  Recall that an ideal $\mathcal{J}\subseteq\mathcal{P}(\omega)$ is a
  \emph{P-ideal} if for each countable sequence $A_n\in\mathcal{J}$ ($n < \omega$) there
  is an $A\in\mathcal{J}$ such that for all $n < \omega$, $A_n \subseteq^* A$.
  \begin{lemma}
    \label{p-ideal}
    Assume OCA+MA$_{\aleph_1}$.  If $\mathcal{I}$ is a dense P-ideal then $\varphi$ is
    trivial.
  \end{lemma}

  \begin{lemma}
    \label{adf}
    Assume $\mathfrak{b} > \aleph_1$.  If $\mathcal{I}$ is not a dense P-ideal, then there is an
    uncountable almost disjoint family $\mathcal{A} \subseteq \mathcal{P}(\omega)$ which is
    disjoint from $\mathcal{I}$.
  \end{lemma}

  \begin{lemma}
    \label{ccc.over.fin}
    Assume OCA.  Let $\mathcal{A}$ be an uncountable, tree-like, almost-disjoint family
    of subsets of $\omega$.  Then $\mathcal{I}\setminus\mathcal{A}$ is countable.
  \end{lemma}

  Theorem~\ref{trivial.embedding} now follows from a straightforward combination
  of Lemmas~\ref{ma}, \ref{p-ideal}, \ref{adf} and \ref{ccc.over.fin}.  To
  illustrate the kind of modifications necessary in translating from
  \cite{Velickovic:OCAA}, we will give a proof of Lemma~\ref{p-ideal}.

  \begin{proof}[Proof of Lemma~\ref{p-ideal}]
    For each $a\in\mathcal{I}$, we fix $Z_a\in\mathcal{C}(X)$ and a continuous, compact-to-one
    map $e_a : Z_a\to a$ such that $\varphi([a]) = [Z_a]$ and for all $b\subseteq
    a$, $\varphi([b]) = [e_a^{-1}(b)]$.  We define $f_a : \omega\to \mathcal{C}(X)$ by
    \[
      f_a(n) = e_a^{-1}(\{n\})
    \]
    Define a partition $[\mathcal{I}]^2 = M_0\cup M_1$ by placing $\{a,b\}\in M_0$ if
    and only if there is some $n\in a\cap b$ such that $f_a(n) \neq f_b(n)$.
    Then $M_0$ is open when $\mathcal{I}$ is given the topology obtained by identifying
    $a\in \mathcal{I}$ with $(a,f_a)\in \mathcal{P}(\omega)\times \fexp{\mathcal{C}(X)}{\omega}$.
    \begin{claim}
      There is no uncountable, $M_0$-homogeneous subset $H$ of $\mathcal{I}$.
    \end{claim}
    \begin{proof}
      Assume $H$ is such a set, and that $|H| = \aleph_1$.  Since $\mathcal{I}$ is a
      P-ideal, there is a set $\bar{H}\subseteq \mathcal{I}$ such that for every $a\in
      H$ there is some $b\in \bar{H}$ with $a\subseteq^* b$, and moreover
      $\bar{H}$ is a chain of order-type $\omega_1$ with respect to
      $\subseteq^*$.  By OCA, there is an uncountable subset of $\bar{H}$ which
      is homogeneous for one of the two colors $M_0$ and $M_1$; hence, by
      passing to this subset, we may assume $\bar{H}$ is either $M_0$ or $M_1$
      homogeneous.

      Say $\bar{H}$ is $M_1$-homogeneous.  Put $\bar{a} = \bigcup \bar{H}$, and
      $\bar{f} = \bigcup_{a\in\bar{H}} f_a$.  Then $\bar{f} : \bar{a}\to\mathcal{C}(X)$,
      and for all $a\in H$ we have $a\subseteq^* \bar{a}$ and $f_{\bar{a}}\rs
      (a\cap\bar{a}) =^* f_a\rs (a\cap\bar{a})$.  Choose $n$ so that for
      uncountably many $a\in H$, we have $a\setminus n \subseteq \bar{a}$, and
      $f_{\bar{a}}\rs a\setminus n = f_a\rs a\setminus n$.  Then if $a,b\in H$ are such, and
      $f_a\rs n = f_b\rs n$, we have $\{a,b\}\in M_1$, a contradiction.
      
      So $\bar{H}$ is $M_0$-homogeneous.  Define a poset $\mathbb{P}$ as follows.  Put
      $p\in \mathbb{P}$ if and only if $p = (A_p, m_p, H_p)$ where $m_p < \omega$,
      $A_p\in\mathcal{C}(K_{m_p})$, and $H_p\in [\bar{H}]^{<\omega}$, and for all
      distinct $a,b\in H_p$, there is an $n\in a\cap b$ such that
      \[
        \lnot (f_a(n)\cap A_p = \emptyset \iff f_b(n)\cap A_p = \emptyset)
      \]
      That is, one of $f_a(n)$, $f_b(n)$ is disjoint from $A_p$, and the other
      isn't.  Put $p \le q$ if and only if $m_p \ge m_q$, $A_p\cap K_{m_q} =
      A_q$, and $H_p\supseteq H_q$.

      First we must show that $\mathbb{P}$ is ccc.  Suppose $\mathcal{X}$ is an uncountable
      subset of $\mathbb{P}$.  We may assume without loss of generality that for some
      fixed $m$ and $A\in\mathcal{C}(K_m)$, and for all $p\in\mathcal{X}$, $m_p = m$ and $A_p =
      A$, and moreover that $H_p$ is the same size for all $p\in\mathcal{X}$.  Let $a_p$
      be the minimal element of $H_p$ under $\subseteq^*$, for each $p\in\mathcal{X}$.
      Find $n_p$ so that for all $a\in H_p$,
      \[
        f_{a_p}\rs (a_p\setminus n_p) \subseteq f_a\qquad e_{a_p}''K_m \subseteq n_p
      \]
      We may assume that for some fixed $n$, we have $n_p = n$ for all $p\in\mathcal{X}$.
      Find $p,q\in\mathcal{X}$ with $f_{a_p}\rs n = f_{a_q}\rs n$.  Since $\{a_p,a_q\}\in
      M_0$, there is some $k\in a_p\cap a_q$ such that $f_{a_p}(k) \neq
      f_{a_q}(k)$.  Then $k\ge n$, and so $f_{a_p}(k)\cap K_m = f_{a_q}(k)\cap
      K_m = \emptyset$.  At least one of $f_{a_p}(k)\setminus f_{a_q}(k)$ and
      $f_{a_q}(k)\setminus f_{a_p}(k)$ must be nonempty; whichever one it is, call it
      $B$.  Put $A_r = A\cup B$ and $H_r = H_p\cup H_q$, and choose $m_r$ large
      enough that $A_r\subseteq K_{m_r}$. Then $r = (A_r,m_r,H_r)\in\mathbb{P}$, and
      $r\le p,q$.

      By MA$_{\aleph_1}$, there is a set $A\in\mathcal{C}(X)$ and an uncountable $H^*
      \subseteq \bar{H}$ such that for all distinct $a,b\in H^*$,
      \[
        \exists n\in a\cap b\quad \lnot (f_a(n) \cap A = \emptyset \iff
        f_b(n)\cap A = \emptyset)
      \]
      Fix $x\subseteq\omega$ such that $F(x) = A$.  Then for all $a\in H^*$,
      $e_a^{-1}(x\cap a) \Delta (A\cap F(a))$ is compact; hence there are $k_a$
      and $m_a$ such that
      \[
        e_a^{-1}(x\cap a \setminus k_a) = (A\cap F(a))\setminus
        K_{m_a}\qquad\mbox{and}\qquad e_a^{-1}(a\setminus k_a) = F(a)\setminus K_{m_a}
      \]
      Then, for all $n\in a\setminus k_a$, $n\in x$ implies $f_a(n) \subseteq A$, and
      $n\not\in x$ implies $f_a(n) \cap A = \emptyset$.  Fix distinct $a,b\in
      H^*$ with $k_a = k_b = k$, and $f_a\rs k = f_b\rs k$.  Then,
      \[
        \forall n\in a\cap b\; (f_a(n)\cap A = \emptyset \iff f_b(n)\cap A =
        \emptyset)
      \]
      This contradicts the choice of $A$.
    \end{proof}
    By OCA, there is a cover of $\mathcal{I}$ by countably many sets $\mathcal{I}_n$, each of
    which is $M_1$-homogeneous.  Since $\mathcal{I}$ is a P-ideal, at least one of the
    $\mathcal{I}_n$'s must be cofinal in $\mathcal{I}$ with respect to $\subseteq^*$.  Choose
    such an $\mathcal{I}_n$, and let $f = \bigcup\set{f_a}{a\in\mathcal{I}_n}$.  Then $f$ is a
    function from some subset of $\omega$ to $\mathcal{C}(X)$.  Setting $e(x) = n$ if
    and only if $x\in f(n)$, we get a function $e : X\to \omega$, and since
    $\mathcal{I}$ is dense and $\mathcal{I}_n$ cofinal in $\mathcal{I}$, $a\mapsto e^{-1}(a)$ witnesses
    that $\varphi$ is trivial.
  \end{proof}

  \section{Coherent families of continuous functions}
  \label{sec:Families}

  \begin{theorem}
    \label{trivial.isomorphism}
    Assume OCA+MA$_{\aleph_1}$.  Let $X$ and $Y$ be zero-dimensional, locally
    compact Polish spaces, and let $\varphi : \mathcal{C}(Y)/\mathcal{K}(Y) \to \mathcal{C}(X)/\mathcal{K}(X)$ be an
    isomorphism.  Then there are compact-open $K\subseteq X$ and $L\subseteq Y$,
    and a homeomorphism $e : X\setminus K\to Y\setminus L$, such that for all $A\in \mathcal{C}(Y\setminus
    L)$, $\varphi([A]) = [e^{-1}(A)]$.
  \end{theorem}

  By Stone duality, a homeomorphism $\varphi : X^*\to Y^*$ induces an isomorphism
  $\hat{\varphi} : \mathcal{C}(Y)/\mathcal{K}(Y)\to \mathcal{C}(X)/\mathcal{K}(X)$, and any map $e$ as in the
  conclusion to Theorem~\ref{trivial.isomorphism} will in this case be a witness
  to the triviality of $\varphi$.  Hence Theorem~\ref{trivial.isomorphism} implies
  Theorem~\ref{main}.  Before proving Theorem~\ref{trivial.isomorphism} we note
  a corollary involving definable isomorphisms.
  \begin{corollary}
    \label{definable}
    Suppose $X$ and $Y$ are zero-dimensional, locally compact, Polish spaces,
    and $\varphi : \mathcal{C}(Y)/\mathcal{K}(Y)\to \mathcal{C}(X)/\mathcal{K}(X)$
    is an isomorphism such that the set
    \[
      \Gamma = \set{(A,B)\in\mathcal{C}(Y)\times \mathcal{C}(X)}{\varphi([A]) = [B]}
    \]
    is Borel.  Then $\varphi$ is trivial.
  \end{corollary}
  \begin{proof}[Proof of Corollary~\ref{definable}]
    The fact that $\varphi$ is an isomorphism between $\mathcal{C}(Y)/\mathcal{K}(Y)$ and
    $\mathcal{C}(X)/\mathcal{K}(X)$ can be written as a $\mathbf{\Pi}^1_2$ statement using
    $\Gamma$; hence by Schoenfield absoluteness, if $V^\mathbb{P}$ is a forcing
    extension satisfying OCA+MA$_{\aleph_1}$ (see \cite{Todorcevic:PPIT}), then
    in $V^\mathbb{P}$ the map $\bar{\varphi} : \mathcal{C}(Y)/\mathcal{K}(Y)\to \mathcal{C}(X)/\mathcal{K}(X)$, defined
    from the reinterpretation of $\Gamma$ in $V^\mathbb{P}$, is also an isomorphism.
    By Theorem~\ref{trivial.isomorphism}, then, we have in $V^\mathbb{P}$ that
    \[
      \exists e\in C(X,Y) \; \forall A\in\mathcal{C}(Y)\; \bar{\varphi}([A]) = [e^{-1}(A)]
    \]
    where $C(X,Y)$ denotes the space of continuous maps from $X$ to $Y$.
    This can be written as a $\mathbf{\Sigma}^1_2$ statement and so by
    Schoenfield absoluteness again, it must be true in $V$ with $\varphi$
    replacing $\bar{\varphi}$.
  \end{proof}

  Before the proof of Theorem~\ref{trivial.isomorphism} we set down some more
  notation.  Fix $X,Y$ and $\varphi$ as in the statement of the theorem.  Let
  $L_n$ be an increasing sequence of compact subsets of $Y$, with union $Y$, and
  let $Y_{n+1} = L_{n+1}\setminus L_n$ and $Y_0 = L_0$.  Let $\mathcal{B}$ be a countable base
  for $Y$ consisting of compact-open sets, such that
  \begin{itemize}
    \item  for all $U\in\mathcal{B}$, the set of $V\in\mathcal{B}$ with $V\supseteq U$ is finite
    and linearly ordered by $\subseteq$, and
    \item  for all $U\in\mathcal{B}$ and all $n < \omega$, either $U\subseteq Y_n$ or
    $U\cap Y_n = \emptyset$.
  \end{itemize}
  It follows that for all $U,V\in\mathcal{B}$, either $U\cap V = \emptyset$, $U\subseteq
  V$, or $V\subseteq U$.  Let $\mathbb{P}$ be the poset of all partitions of $Y$ into
  elements of $\mathcal{B}$, ordered by refinement;
  \[
    P \prec Q \iff \forall U\in P\; \exists V\in Q\quad U\subseteq V
  \]
  We also use $\prec^*$ to denote \emph{eventual refinement};
  \[
    P \prec^* Q \iff \forall^\infty U\in P\; \exists V\in Q\quad U\subseteq V
  \]
  When $P\prec^* Q$ we let $\Gamma(P,Q)$ be the least $n$ such that every $U\in
  P$ disjoint from $L_n$ is contained in some element of $Q$.

  For a given $P\in\mathbb{P}$, let $s_P : Y\to P$ be the unique function satisfying
  $x\in s_P(x)$ for all $x\in Y$; similarly, when $P,Q\in\mathbb{P}$ and $P\prec Q$ we
  let $s_{PQ} : P\to Q$ be the unique function satisfying $U\subseteq s_{PQ}(U)$
  for all $U\in P$.  These maps induce embeddings $\sigma_P : \mathcal{P}(P)/\fin\to
  \mathcal{C}(Y)/\mathcal{K}(Y)$ and $\sigma_{PQ} : \mathcal{P}(Q)/\fin\to\mathcal{P}(P)/\fin$ in the usual
  way.

  Finally, we need to prove a uniqueness result for maps $e : Z\to\omega$
  inducing the same map $\mathcal{P}(\omega)/\fin\to \mathcal{C}(Z)/\mathcal{K}(Z)$.
  \begin{lemma}
    \label{lift.uniqueness} Suppose $Z\in\mathcal{C}(X)$ and $e, f : Z\to \omega$ are
    continuous, compact-to-one maps, such that $e^{-1}(a)\Delta f^{-1}(a)$ is
    compact for every $a\subseteq\omega$.  Then $\set{x\in Z}{e(x)\neq f(x)}$ is
    compact.
  \end{lemma}
  \begin{proof}
    Suppose not; then for some infinite set $I\subseteq \omega$ and all $n\in
    I$, there is a point $x_n\in Z\cap X_n$ such that $e(x_n) \neq f(x_n)$.
    Since $e$ and $f$ are compact-to-one, we may assume also that $m\neq n$
    implies $e(x_m) \neq e(x_n)$ and $f(x_m) \neq f(x_n)$.  Now define a
    coloring $F : [I]^2\to 3$ by
    \[
      F(\{m < n\}) = \left\{
        \begin{array}{ll}
          0 & e(x_m) \neq f(x_n) \land f(x_m) \neq e(x_n) \\
          1 & e(x_m) = f(x_n) \land f(x_m)\neq e(x_n) \\
          2 & e(x_m) \neq f(x_n) \land f(x_m) = e(x_n) \\
        \end{array}
        \right.
    \]
    By Ramsey's theorem, there is an infinite set $a\subseteq I$ which is
    homogeneous for this coloring.  Suppose first that $a$ is $1$-homogeneous,
    and let $m < n < k$ be members of $a$.  Then
    \[
      e(x_m) = f(x_n)\quad\mbox{and}\quad e(x_m) = f(x_k)\quad\mbox{and}\quad
      e(x_n) = f(x_k)
    \]
    which implies $e(x_n) = f(x_n)$, a contradiction.  Similarly, $a$ cannot be
    $2$-homogeneous.

    Now suppose $a$ is $0$-homogeneous.  Let $a = a_0\cup a_1$ be a partition of
    $a$ into two infinite sets, and put $W_i = \set{x_n}{n\in a_i}$ and $W =
    \set{x_n}{n\in a} = W_0\cup W_1$.  From the homogeneity of $a$, it follows
    that $e''W\cap f''W = \emptyset$, and hence (as $e$ and $f$ are injective on
    $W$)
    \[
      W\cap e^{-1}((e''W_0)\cup (f''W_1)) = W_0\quad\mbox{and}\quad W\cap
      f^{-1}((e''W_0)\cup (f''W_1)) = W_1
    \]
    So, if $b = e''W_0\cup f''W_1$, we have $W\subseteq e^{-1}(b)\Delta
    f^{-1}(b)$.  But $W$ is not compact, so this is a contradiction.
  \end{proof}

  \begin{proof}[Proof of Theorem~\ref{trivial.isomorphism}]
    For each $P\in\mathbb{P}$, let $\varphi_P = \varphi\circ \sigma_P$.  Then
    $\varphi_P$ is an embedding of $\mathcal{P}(P)/\fin$ into
    $\mathcal{C}(X)/\mathcal{K}(X)$.  By Theorem~\ref{trivial.isomorphism},
    there is a continuous map $e_P : X\to P$ such that $a\mapsto e_P^{-1}(a)$
    lifts $\varphi_P$.  Note that if $P,Q\in\mathbb{P}$ and $P\prec^* Q$, then
    the following diagram commutes:

    \begin{center}
    \leavevmode
    \xymatrix{
      \mathcal{P}(P)/\fin \ar[r]^{\varphi_P} & \mathcal{C}(X)/\mathcal{K}(X) \\
      \mathcal{P}(Q)/\fin \ar[u]^{\sigma_{PQ}} \ar[ur]_{\varphi_Q} & \\
    }
    \end{center}

%    \[
%      \begin{tikzpicture}
%      \matrix (m) [cdg.widematrix] {
%        \mathcal{P}(P)/\fin & \mathcal{C}(X)/\mathcal{K}(X) \\
%        \mathcal{P}(Q)/\fin & \\
%      };
%      \path [cdg.path]
%        (m-1-1) edge node[auto]{$\varphi_P$} (m-1-2)
%        (m-2-1) edge node[below]{$\varphi_Q$} (m-1-2)
%        (m-2-1) edge node[auto]{$\sigma_{PQ}$} (m-1-1);
%      \end{tikzpicture}
%    \]
    So by Lemma~\ref{lift.uniqueness}, the set $\set{x\in X}{s_{PQ}(e_P(x)) \neq
    e_Q(x)}$ is compact.  Now let $[\mathbb{P}]^2 = M_0\cup M_1$ be the
    partition defined by
    \[
      \{P,Q\}\in M_0\iff \exists x\in X\quad s_{P,P\lor Q}(e_P(x)) \neq
      s_{Q,P\lor Q}(e_Q(x)) 
    \]
    Here $P\lor Q$ is the finest partition coarser than both $P$ and $Q$.  If
    we define $f_P : \mathcal{B}\to\mathcal{C}(X)$ by
    \[
      f_P(U) = \set{x\in X}{e_P(x) \subseteq U}
    \]
    then we have
    \[
      \{P,Q\}\in M_0\iff \exists U\in\mathcal{B}\quad f_P(U) \neq f_Q(U)
    \]
    and it follows that $M_0$ is open in the topology on $\mathbb{P}$ obtained by
    identifying $P$ with $f_P$.

    \begin{claim}
      There is no uncountable, $M_0$-homogeneous subset of $\mathbb{P}$.
    \end{claim}
    \begin{proof}
      Suppose $H$ is such, and has size $\aleph_1$.  Using $MA_{\aleph_1}$ with
      a simple modification of Hechler forcing, we see that there is some
      $\bar{P}\in\mathbb{P}$ such that $P\succ^*\bar{P}$ for all $P\in H$.  By thinning
      out $H$ and refining a finite subset of $\bar{P}$, we may assume that
      $P\succ\bar{P}$ for all $P\in H$, and moreover that there is an $\bar{n}$
      such that for all $P\in H$,
      \[
        \set{x\in X}{s_{\bar{P},P}(e_{\bar{P}}(x)) \neq e_P(x)} \subseteq K_{\bar{n}}
      \]
      Now fix $P,Q\in H$ such that $e_P\rs K_{\bar{n}} = e_Q\rs K_{\bar{n}}$.
      Then $s_{P,P\lor Q}\circ e_P = s_{Q,P\lor Q}\circ e_Q$, contradicting
      the fact that $\{P,Q\}\in M_0$.
    \end{proof}

    By OCA, there is a countable cover of $\mathbb{P}$ by $M_1$-homogeneous sets; since
    $\mathbb{P}$ is countably directed under $\succ^*$, it follows that one of them,
    say $\mathbb{Q}$, is cofinal in $\mathbb{P}$.  It follows moreover that for some $n$, we
    have
    \[
      \forall P\in\mathbb{P}\;\exists Q\in\mathbb{Q}\quad \Gamma(Q,P)\le n
    \]
    That is, $\mathbb{Q}$ is cofinal in $\mathbb{P}$ under $\succ^n$ defined by
    \[
      P \prec^n Q \iff \forall U\in P \;\; (U\cap L_n = \emptyset \implies
      \exists V\in Q\; U\subseteq V)
    \]
    
    \begin{claim}
      There is a compact set $K\subseteq X$ and a unique continuous map $e :
      X\setminus K\to Y$ satisfying
      \[
        \forall x\in X\setminus K\quad e(x)\in \bigcap_{P\in\mathcal{Q}} e_P(x)
      \]
    \end{claim}
    \begin{proof}
      Fix $x\in X$.  If $P,Q\in\mathbb{Q}$, then by $M_1$-homogeneity of
      $\mathbb{Q}$ we have
      \[
        s_{P,P\lor Q}(e_P(x)) = s_{Q,P\lor Q}(e_Q(x))
      \]
      Then, the unique member of $P\lor Q$ containing $e_P(x)$ is the same as
      the unique member of $P\lor Q$ containing $e_Q(x)$.  It follows that
      $e_P(x)\cap e_Q(x)\neq\emptyset$, and so either $e_P(x)\subseteq e_Q(x)$
      or vice-versa.  Then the collection $\set{e_P(x)}{P\in\mathbb{Q}}$ is a chain,
      and hence by compactness has nonempty intersection.

      Now let
      \[
        K = \set{x\in X}{\forall P\in\mathbb{Q}\; e_P(x)\subseteq L_n} \subseteq
        \bigcap_{P\in\mathbb{Q}} e_P^{-1}(P\cap \mathcal{C}(L_n))
      \]
      Then $K$ is contained in a compact set.  If $x\in X\setminus K$ and
      $P\in\mathbb{Q}$, then $e_P(x)$ is disjoint from $L_n$.  Then for any
      $x\in X\setminus K$ and $\epsilon > 0$, there is some $P\in\mathbb{Q}$
      such that $e_P(x)$ has diameter less than $\epsilon$ (since $\mathbb{Q}$
      is cofinal in $\mathbb{P}$ under $\succ^n$).  Thus $e$, as defined above,
      is unique.

      To see that $e$ is continuous, note that for any open $U\subseteq X$,
      \[
        x\in e^{-1}(U) \iff \exists P\in\mathbb{Q}\quad e_P(x)\subseteq U
      \]
    \end{proof}

    \begin{claim}
      The map $U\mapsto e^{-1}(U)$ lifts $\varphi$.
    \end{claim}
    \begin{proof}
      Fix $P\in\mathbb{Q}$, and let $U\in P$.  Then clearly, for all $x\in
      X\setminus K$, $e_P(x) = U$ if and only if $e(x)\in U$.  Since there are
      only finitely many $U\in P$ such that one of $e_P^{-1}(\{U\})$ or
      $e^{-1}(U)$ meets $K$, it follows that
      \[
        \forall^\infty U\in P\; e_P^{-1}(\{U\}) = e^{-1}(U)
      \]
      Then $U\mapsto e^{-1}(U)$ lifts $\varphi_P$.

      Now fix $A\in\mathcal{C}(Y)$.  Then there is some $P\in\mathbb{P}$ such
      that $A$ can be written as a union of a subset of $P$.  Find
      $Q\in\mathbb{Q}$ with $Q\prec^* P$; then, up to a compact set, $A$ can be
      written as a union of some subset $a$ of $Q$.  Hence,
      \[
        \varphi [A] = \varphi_Q [a] = [e^{-1}(A)]
      \]
    \end{proof}

  \end{proof}

\bibliographystyle{amsplain}

\end{document}